\documentclass[12pt,fleqn]{article}

\usepackage{amsmath,amssymb,amsthm}
\usepackage{hyperref}

\newcommand{\as}{\ins{as}}
\newcommand{\bgset}[1]{\big\{#1\big\}}
\newcommand{\calG}{{\cal G}}
\newcommand{\dint}{\ds{\int}}
\newcommand{\ds}[1]{\displaystyle #1}
\newcommand{\dualp}[3][]{\left(#2,#3\right)_{#1}}
\newcommand{\half}{\frac{1}{2}}
\newcommand{\hquad}{\hspace{0.08in}}
\newcommand{\ins}[1]{\hquad \text{#1} \hquad}
\newcommand{\ip}[3][]{\left(#2,#3\right)_{#1}}
\newcommand{\isom}{\approx}
\newcommand{\norm}[2][]{\left\|#2\right\|_{#1}}
\renewcommand{\o}{\text{o}}
\newcommand{\PS}[1]{$(\text{PS})_{#1}$}
\newcommand{\QED}{\mbox{\qedhere}}
\newcommand{\R}{\mathbb R}
\newcommand{\restr}[2]{\left.#1\right|_{#2}}
\newcommand{\seq}[1]{\left(#1\right)}
\newcommand{\set}[1]{\left\{#1\right\}}
\newcommand{\sip}[3][]{(#2,#3)_{#1}}
\newcommand{\unifae}{\ins{uniformly a.e.}}

\newtheorem{corollary}{Corollary}[section]
\newtheorem{lemma}[corollary]{Lemma}
\newtheorem{theorem}[corollary]{Theorem}

\theoremstyle{remark}
\newtheorem{example}[corollary]{Example}

\numberwithin{equation}{section}

\title{\bf Some Results for Impulsive Problems via Morse Theory\thanks{{\em MSC2010:} Primary 34B37, Secondary 58E05
\newline \smallskip \indent\; {\em Key Words and Phrases:} impulsive problems, asymptotically piecewise linear problems, resonance set, nontrivial solutions, Morse theory, critical groups, saddle point theorem}
}
\author{\bf Ravi P. Agarwal\\
Department of Mathematics\\
Texas A\&M University\\
Kingsville, TX 78363-8202, USA\\
\em agarwal@tamuk.edu\\
[\medskipamount]
\bf T. Gnana Bhaskar and Kanishka Perera\\
Department of Mathematical Sciences\\
Florida Institute of Technology\\
Melbourne, FL 32901, USA\\
\em gtenali@fit.edu \& kperera@fit.edu}
\date{}

\begin{document}

\maketitle

\begin{abstract}
We use Morse theory to study impulsive problems. First we consider asymptotically piecewise linear problems with superlinear impulses, and prove a new existence result for this class of problems using the saddle point theorem. Next we compute the critical groups at zero when the impulses are asymptotically linear near zero, in particular, we identify an important resonance set for this problem. As an application, we finally obtain a nontrivial solution for asymptotically piecewise linear problems with impulses that are asymptotically linear at zero and superlinear at infinity. Our results here are based on the simple observation that the underlying Sobolev space naturally splits into a certain finite dimensional subspace where all the impulses take place and its orthogonal complement that is free of impulsive effects.
\end{abstract}

\section{Introduction}

Impulsive problems arise naturally in studies of evolutionary processes that involve abrupt changes in the state of the system, triggered by instantaneous perturbations called impulses. Examples include games where players can affect the game only at discrete instants (see Chikrii, Matychyn, and Chikrii \cite{MR2341237}), two person zero sum games with separated impulsive dynamics (see Cr{\"u}ck, Quincampoix, and Saint-Pierre \cite{MR2341245}), pulse vaccination strategy (see Stone, Shulgin, and Agur \cite{MR1756756}), and optimal impulsive harvesting (see Zhang, Shuai, and Wang \cite{MR1962293}). Classical approaches to such problems include fixed point theory (see, e.g., Lin and Jiang \cite{MR2241134}) and the method of upper and lower solutions (see, e.g., Liu and Guo \cite{MR1487265}). More recently, variational methods have been widely used to study impulsive problems (see, e.g., Tian and Ge \cite{MR2465922}, Nieto and O'Regan \cite{MR2474254}, Zhou and Li \cite{MR2532812}, Zhang and Yuan \cite{MR2570535}, Zhang and Li \cite{MR2570525}, Bai and Dai \cite{MR2782888}, Han and Wang \cite{MR2781070}, and Gong, Zhang, and Tang \cite{MR2830975}).

In this paper we use Morse theory to study impulsive problems. First we consider asymptotically piecewise linear problems with superlinear impulses. Although asymptotically piecewise linear nonlinearities are quite natural in this setting, they do not seem to have been studied in the literature. We will prove a new existence result for this class of problems using the saddle point theorem. Next we compute the critical groups at zero when the impulses are asymptotically linear near zero. In particular, we will identify an important resonance set for this problem. The effect of impulses on critical groups has not been studied previously, to the best of our knowledge. As an application, we finally obtain a nontrivial solution for asymptotically piecewise linear problems with impulses that are asymptotically linear at zero and superlinear at infinity. Our results here are based on the simple observation that the underlying Sobolev space naturally splits into a certain finite dimensional subspace where all the impulses take place and its orthogonal complement that is free of impulsive effects.

Let $m$ be a positive integer, let $0 = x_0 < x_1 < \cdots < x_m < x_{m+1} = 1$, and consider the impulsive problem
\begin{equation} \label{1.1}
    \left\{\begin{aligned}
    - u'' & = f(x,u), \quad x \in (0,1) \setminus \set{x_1,\dots,x_m}\\[10pt]
    u(0) & = u(1) = 0, \qquad u(x_j^+) = u(x_j^-), \quad j = 1,\dots,m\\[10pt]
    u'(x_j^+) & = u'(x_j^-) - \imath_j(u(x_j)), \quad j = 1,\dots,m,
    \end{aligned}\right.
\end{equation}
where $f$ is a Carath\'{e}odory function on $(0,1) \times \R$,
\begin{equation*}
    u(x_j^\pm) = \lim_{\stackrel{x \to x_j}{x \gtrless x_j}} u(x), \qquad u'(x_j^\pm) = \lim_{\stackrel{x \to x_j}{x \gtrless x_j}} u'(x),
\end{equation*}
and $\imath_j$ are continuous functions on $\R$. Denoting by $H^1_0(0,1)$ the usual Sobolev space with the inner product
\begin{equation*}
    \ip{u}{v} = \int_0^1 u' v',
\end{equation*}
a weak solution of \eqref{1.1} is a function $u \in H^1_0(0,1)$ such that
\begin{equation*}
    \int_0^1 u' v' = \int_0^1 f(x,u)\, v + \sum_{j=1}^m \imath_j(u(x_j))\, v(x_j) \quad \forall v \in H^1_0(0,1).
\end{equation*}
Noting that $H^1_0(0,1)$ is continuously imbedded in $C[0,1]$, we see that weak solutions coincide with the critical points of the $C^1$-functional
\begin{equation*}
    \Phi(u) = \half \int_0^1 (u')^2 - \int_0^1 F(x,u) - \sum_{j=1}^m I_j(u(x_j)), \quad u \in H = H^1_0(0,1),
\end{equation*}
where
\begin{equation*}
    F(x,t) = \int_0^t f(x,s)\, ds, \qquad I_j(t) = \int_0^t \imath_j(s)\, ds
\end{equation*}
are the primitives of $f$ and $\imath_j$, respectively.

The closed linear subspace
\begin{equation*}
    N = \bgset{u \in H : u(x_j) = 0,\, j = 1,\dots,m}
\end{equation*}
is important here since each $I_j(0) = 0$. For $j = 1,\dots,m$, the mapping $H \to \R,\, u \mapsto u(x_j)$ is a bounded linear functional on $H$ and hence there is a unique $w_j \in H$ such that $u(x_j) = \ip{u}{w_j}$ by the Riesz-Frechet representation theorem. In fact,
\begin{equation} \label{1.2}
    w_j(x) = \begin{cases}
    (1 - x_j)\, x, & 0 \le x \le x_j\\[10pt]
    x_j\, (1 - x), & x_j \le x \le 1.
    \end{cases}
\end{equation}
Since $x_j$ are distinct, $w_j$ are linearly independent, so $N$ is the orthogonal complement of the $m$-dimensional subspace $M$ that they span. Hence we have the orthogonal decomposition
\begin{equation*}
    H = N \oplus M, \quad u = v + w,
\end{equation*}
and
\begin{equation} \label{1.3}
    \Phi(u) = \half \int_0^1 \left((v')^2 + (w')^2\right) - \int_0^1 F(x,u) - \sum_{j=1}^m I_j(w(x_j)).
\end{equation}
We will make use of this splitting throughout the paper.

By \eqref{1.2}, each $w \in M$ is affine on the subintervals $[x_{j-1},x_j]$. Since the space of continuous functions on $[0,1]$ that are affine on these subintervals and vanish at the endpoints is also $m$-dimensional, it follows that $M$ is precisely this subspace. Then we also have
\begin{equation*}
    \max_{x \in [0,1]}\, |w(x)| = \max_{j = 1,\dots,m}\, |w(x_j)| \quad \forall w \in M,
\end{equation*}
and this is an equivalent norm on this finite dimensional space.

The subspace $N$ has the decomposition
\begin{equation*}
    N = \bigoplus_{j=1}^{m+1} N_j, \quad v = \sum_{j=1}^{m+1} v_j
\end{equation*}
where $N_j = H^1_0(x_{j-1},x_j)$, $v_j = \chi_j\, v$, and
\begin{equation*}
    \chi_j(x) = \begin{cases}
    1, & x \in (x_{j-1},x_j)\\[10pt]
    0, & x \in (0,1) \setminus (x_{j-1},x_j)
    \end{cases}
\end{equation*}
is the characteristic function of the subinterval $(x_{j-1},x_j)$. Combining this with \eqref{1.3} gives
\begin{equation} \label{1.4}
    \Phi(u) = \half \left[\sum_{j=1}^{m+1} \int_{x_{j-1}}^{x_j} (v_j')^2 + \int_0^1 (w')^2\right] - \int_0^1 F(x,u) - \sum_{j=1}^m I_j(w(x_j)).
\end{equation}
We will make use of this splitting in the next section.

\section{Asymptotically Piecewise Linear Problems with Superlinear Impulses}

In this section we assume that $f$ is asymptotically piecewise linear in the sense that
\begin{equation} \label{2.1}
    f(x,t) = \sum_{j=1}^{m+1} a_j\, \chi_j(x)\, t + g(x,t)
\end{equation}
where $a_1,\dots,a_{m+1} \in \R$ and $g$ satisfies
\begin{equation} \label{2.2}
    |g(x,t)| \le C \left(|t|^{r-1} + 1\right) \quad \text{for a.e. } x \in (0,1) \text{ and all } t \in \R
\end{equation}
for some $r \in (1,2)$ and a generic positive constant $C$. For the sake of simplicity we will only consider the nonresonant case where, for all $j$, $a_j$ is not in the set
\begin{equation*}
    \sigma_j = \set{\lambda^j_k = \frac{k^2 \pi^2}{(x_j - x_{j-1})^2} : k = 1,2,\dots}
\end{equation*}
of eigenvalues of the problem
\begin{equation*}
    \left\{\begin{aligned}
    - u'' & = \lambda\, u, \quad x \in (x_{j-1},x_j)\\[10pt]
    u(x_{j-1}) & = u(x_j) = 0.
    \end{aligned}\right.
\end{equation*}
Regarding the impulses we assume the superlinearity conditions
\begin{equation} \label{2.3}
    t\, \imath_j(t) \ge c\, |t|^\mu - C \quad \forall t \in \R,\, j = 1,\dots,m
\end{equation}
for some $\mu > 2$ and $c > 0$. The main result of this section is

\begin{theorem} \label{Theorem 2.1}
If \eqref{2.1} -- \eqref{2.3} hold, and $a_j \notin \sigma_j$ for $j = 1,\dots,m+1$, then problem \eqref{1.1} has a solution.
\end{theorem}

By \eqref{2.1} and \eqref{2.2},
\begin{equation*}
    F(x,t) = \sum_{j=1}^{m+1} \half\, a_j\, \chi_j(x)\, t^2 + G(x,t)
\end{equation*}
where $G(x,t) = \dint_0^t g(x,s)\, ds$ satisfies
\begin{equation} \label{2.4}
    |G(x,t)| \le C \left(|t|^r + 1\right) \quad \text{for a.e. } x \in (0,1) \text{ and all } t \in \R.
\end{equation}
Combining this with \eqref{1.4} gives
\begin{eqnarray*}
    \Phi(u) & = & \half\, \Bigg[\sum_{j=1}^{m+1} \int_{x_{j-1}}^{x_j} \left((v_j')^2 - a_j\, v_j^2\right) + \int_0^1 (w')^2 - \sum_{j=1}^{m+1} a_j \int_{x_{j-1}}^{x_j} w^2\Bigg]\\[10pt]
    & & - \sum_{j=1}^{m+1} a_j \int_{x_{j-1}}^{x_j} v_j\, w - \int_0^1 G(x,u) - \sum_{j=1}^m I_j(w(x_j)).
\end{eqnarray*}
By \eqref{2.3},
\begin{equation} \label{2.5}
    I_j(t) \ge \tilde{c}\, |t|^\mu - C \quad \forall t \in \R,\, j = 1,\dots,m
\end{equation}
for some $\tilde{c} > 0$.

Let $J_0$ be the set of those $j$ for which $a_j < \lambda^j_1$ and let $J_1 = \set{1,\dots,m+1} \setminus J_0$. For each $j \in J_1$, $\lambda^j_{d_j} < a_j < \lambda^j_{d_j+1}$ for some $d_j \ge 1$, and we have the decomposition
\begin{equation*}
    N_j = N_j^+ \oplus N_j^-, \quad v_j = v_j^+ + v_j^-
\end{equation*}
where $N_j^-$ is the $d_j$-dimensional subspace spanned by the eigenfunctions of $\lambda^j_1,\dots,\lambda^j_{d_j}$ and $N_j^+$ is its orthogonal complement. Then
\begin{eqnarray*}
    \Phi(u) & = & \half\, \Bigg[\sum_{j \in J_0} \int_{x_{j-1}}^{x_j} \left((v_j')^2 - a_j\, v_j^2\right) + \sum_{j \in J_1} \int_{x_{j-1}}^{x_j} \big(({v_j^+}')^2 - a_j\, (v_j^+)^2\big)\\[10pt]
    & & + \sum_{j \in J_1} \int_{x_{j-1}}^{x_j} \big(({v_j^-}')^2 - a_j\, (v_j^-)^2\big) + \int_0^1 (w')^2 - \sum_{j=1}^{m+1} a_j \int_{x_{j-1}}^{x_j} w^2\Bigg]\\[10pt]
    & & - \sum_{j=1}^{m+1} a_j \int_{x_{j-1}}^{x_j} v_j\, w - \int_0^1 G(x,u) - \sum_{j=1}^m I_j(w(x_j))
\end{eqnarray*}
for
\begin{equation} \label{2.6}
    u = \sum_{j \in J_0} v_j + \sum_{j \in J_1}\, (v_j^+ + v_j^-) + w \in \bigoplus_{j \in J_0} N_j \oplus \bigoplus_{j \in J_1}\, (N_j^+ \oplus N_j^-) \oplus M.
\end{equation}
We have
\begin{gather*}
    \int_{x_{j-1}}^{x_j} (v_j')^2 \ge \lambda^j_1 \int_{x_{j-1}}^{x_j} v_j^2, \quad j \in J_0,\\[10pt]
    \int_{x_{j-1}}^{x_j} ({v_j^+}')^2 \ge \lambda^j_{d_j+1} \int_{x_{j-1}}^{x_j} (v_j^+)^2, \hquad \int_{x_{j-1}}^{x_j} ({v_j^-}')^2 \le \lambda^j_{d_j} \int_{x_{j-1}}^{x_j} (v_j^-)^2, \quad j \in J_1,
\end{gather*}
so
\begin{gather}
    \int_{x_{j-1}}^{x_j} \left((v_j')^2 - a_j\, v_j^2\right) \ge c_j \norm{v_j}^2, \quad j \in J_0, \label{2.7}\\[10pt]
    \int_{x_{j-1}}^{x_j} \big(({v_j^+}')^2 - a_j\, (v_j^+)^2\big) \ge c_j^+ \norm{v_j^+}^2, \notag\\[10pt]
    \qquad \int_{x_{j-1}}^{x_j} \big(({v_j^-}')^2 - a_j\, (v_j^-)^2\big) \le - c_j^- \norm{v_j^-}^2, \quad j \in J_1 \label{2.8}
\end{gather}
where the constants
\begin{gather*}
c_j = 1 - \frac{\max \set{a_j,0}}{\lambda^j_1}, \quad j \in J_0,\\[10pt]
c_j^+ = 1 - \frac{a_j}{\lambda^j_{d_j+1}}, \hquad c_j^- = \frac{a_j}{\lambda^j_{d_j}} - 1, \quad j \in J_1
\end{gather*}
are all positive.

Recall that $\Phi$ satisfies the Palais-Smale compactness condition \PS{} if every sequence $\seq{u_n}$ in $H$ such that $\seq{\Phi(u_n)}$ is bounded and $\Phi'(u_n) \to 0$, called a \PS{} sequence, has a convergent subsequence.

\begin{lemma} \label{Lemma 2.2}
If \eqref{2.1} -- \eqref{2.3} hold, and $a_j \notin \sigma_j$ for $j = 1,\dots,m+1$, then every sequence $\seq{u_n}$ in $H$ such that $\Phi'(u_n) \to 0$ has a convergent subsequence, in particular, $\Phi$ satisfies the {\em \PS{}} condition.
\end{lemma}

\begin{proof}
By a standard argument it suffices to show that $\seq{u_n}$ is bounded. Referring to the decomposition \eqref{2.6}, write
\begin{equation*}
    u_n = \sum_{j \in J_0} v_{nj} + \sum_{j \in J_1}\, (v_{nj}^+ + v_{nj}^-) + w_n
\end{equation*}
and set
\begin{equation*}
    \bar{u}_n = \sum_{j \in J_0} v_{nj} + \sum_{j \in J_1}\, (v_{nj}^+ - v_{nj}^-) - w_n.
\end{equation*}
Then
\begin{eqnarray*}
    \dualp{\Phi'(u_n)}{\bar{u}_n} & = & \sum_{j \in J_0} \int_{x_{j-1}}^{x_j} \left((v_{nj}')^2 - a_j\, (v_{nj})^2\right)\\[10pt]
    & & + \sum_{j \in J_1} \int_{x_{j-1}}^{x_j} \left[\big(({v_{nj}^+}')^2 - a_j\, (v_{nj}^+)^2\big) - \big(({v_{nj}^-}')^2 - a_j\, (v_{nj}^-)^2\big)\right]\\[10pt]
    & & - \int_0^1 (w_n')^2 + \sum_{j=1}^{m+1} a_j \int_{x_{j-1}}^{x_j} w_n^2 + 2 \sum_{j \in J_1} a_j \int_{x_{j-1}}^{x_j} v_{nj}^-\, w_n\\[10pt]
    & & - \int_0^1 g(x,u_n)\, \bar{u}_n + \sum_{j=1}^m \imath_j(w_n(x_j))\, w_n(x_j).
\end{eqnarray*}
Since $\Phi'(u_n) \to 0$, this together with \eqref{2.7}, \eqref{2.8}, \eqref{2.2}, and \eqref{2.3} gives
\begin{multline*}
    \sum_{j \in J_0} c_j \norm{v_{nj}}^2 + \sum_{j \in J_1} \big(c_j^+ \norm{v_{nj}^+}^2 + c_j^- \norm{v_{nj}^-}^2\big) + c \sum_{j=1}^m |w_n(x_j)|^\mu\\[10pt]
    \le C \left[\norm{w_n}^2 + \sum_{j \in J_1} \norm{v_{nj}^-} \norm{w_n} + \norm{u_n}^{r-1} \norm{\bar{u}_n} + \norm{\bar{u}_n} + 1\right].
\end{multline*}
Since $\max_j\, |w(x_j)|$ defines an equivalent norm on $M$, $\mu > 2$, $\norm{\bar{u}_n} = \norm{u_n}$, and $r < 2$, boundedness of
\begin{equation*}
    \norm{u_n}^2 = \sum_{j \in J_0} \norm{v_{nj}}^2 + \sum_{j \in J_1} \big(\norm{v_{nj}^+}^2 + \norm{v_{nj}^-}^2\big) + \norm{w_n}^2
\end{equation*}
follows.
\end{proof}

We are now ready to give

\begin{proof}[Proof of Theorem \ref{Theorem 2.1}]
We apply the saddle point theorem to the splitting
\begin{equation*}
    H = \bigg(\bigoplus_{j \in J_1} N_j^- \oplus M\bigg) \oplus \bigg(\bigoplus_{j \in J_0} N_j \oplus \bigoplus_{j \in J_1} N_j^+\bigg) =: H_1 \oplus H_2.
\end{equation*}
By Lemma \ref{Lemma 2.2}, $\Phi$ satisfies the \PS{} condition. For $u = \sum_{j \in J_1} v_j^- + w \in H_1$,
\begin{eqnarray*}
    \Phi(u) & = & \half\, \Bigg[\sum_{j \in J_1} \int_{x_{j-1}}^{x_j} \big(({v_j^-}')^2 - a_j\, (v_j^-)^2\big) + \int_0^1 (w')^2 - \sum_{j=1}^{m+1} a_j \int_{x_{j-1}}^{x_j} w^2\Bigg]\\[10pt]
    & & - \sum_{j \in J_1} a_j \int_{x_{j-1}}^{x_j} v_j^-\, w - \int_0^1 G(x,u) - \sum_{j=1}^m I_j(w(x_j))\\[10pt]
    & \le & - \half \sum_{j \in J_1} c_j^- \norm{v_j^-}^2 - \tilde{c} \sum_{j=1}^m |w(x_j)|^\mu\\[10pt]
    & & + C \left[\norm{w}^2 + \sum_{j \in J_1} \norm{v_j^-} \norm{w} + \norm{u}^r + 1\right]
\end{eqnarray*}
by \eqref{2.8}, \eqref{2.4}, and \eqref{2.5}. Since $\max_j\, |w(x_j)|$ is an equivalent norm on $M$, $\mu > 2$, and $r < 2$, it follows that $\Phi(u) \to - \infty$ as
\begin{equation*}
    \norm{u}^2 = \sum_{j \in J_1} \norm{v_j^-}^2 + \norm{w}^2 \to \infty.
\end{equation*}
On the other hand, for $u = \sum_{j \in J_0} v_j + \sum_{j \in J_1} v_j^+ \in H_2$,
\begin{eqnarray*}
    \Phi(u) & = & \half\, \Bigg[\sum_{j \in J_0} \int_{x_{j-1}}^{x_j} \left((v_j')^2 - a_j\, v_j^2\right) + \sum_{j \in J_1} \int_{x_{j-1}}^{x_j} \big(({v_j^+}')^2 - a_j\, (v_j^+)^2\big)\Bigg]\\[10pt]
    & & - \int_0^1 G(x,u)\\[10pt]
    & \ge & \half \left[\sum_{j \in J_0} c_j \norm{v_j}^2 + \sum_{j \in J_1} c_j^+ \norm{v_j^+}^2\right] - C \left(\norm{u}^r + 1\right)
\end{eqnarray*}
by \eqref{2.7}, \eqref{2.8}, and \eqref{2.4}. Since $r < 2$, it follows that $\Phi$ is bounded from below on $H_2$. Thus, $\Phi$ has a critical point by the saddle point theorem.
\end{proof}

\section{Critical Groups at Zero for Asymptotically Linear Impulses}

Now assume that $f(\cdot,0) = 0$ and $\imath_j(0) = 0,\, j = 1,\dots,m$, so that $u = 0$ is a solution of problem \eqref{1.1}, and recall that the critical groups of $\Phi$ at zero are defined by
\begin{equation} \label{3.1}
    C_q(\Phi,0) = H_q(\Phi^0 \cap U,\Phi^0 \cap U \setminus \set{0}), \quad q \ge 0,
\end{equation}
where $\Phi^0 = \set{u \in H : \Phi(u) \le 0}$, $U$ is any neighborhood of $0$, and $H_\ast(\cdot,\cdot)$ are the relative singular homology groups. In this section we compute them when
\begin{equation} \label{3.2}
    f(x,t) = \o(t) \as t \to 0, \unifae
\end{equation}
and
\begin{equation} \label{3.3}
    \imath_j(t) = b_j\, t + h_j(t), \quad j = 1,\dots,m
\end{equation}
where $b_1,\dots,b_m \in \R$ are such that the asymptotic problem
\begin{equation} \label{3.4}
    \left\{\begin{aligned}
    - u'' & = 0, \quad x \in (0,1) \setminus \set{x_1,\dots,x_m}\\[10pt]
    u(0) & = u(1) = 0, \qquad u(x_j^+) = u(x_j^-), \quad j = 1,\dots,m\\[10pt]
    u'(x_j^+) & = u'(x_j^-) - b_j\, u(x_j), \quad j = 1,\dots,m
    \end{aligned}\right.
\end{equation}
has only the trivial solution and
\begin{equation} \label{3.5}
    h_j(t) = \o(t) \as t \to 0,\, j = 1,\dots,m.
\end{equation}

Let $B$ be the set of those points $b = (b_1,\dots,b_m) \in \R^m$ for which problem \eqref{3.4} has a nontrivial solution. We will call $B$ the resonance set for this problem. Clearly, the solution set of the equations on the first two lines of \eqref{3.4} is precisely the subspace $M$. Since $\set{w_1,\dots,w_m}$, where $w_j$ is given by \eqref{1.2}, is a basis of $M$, it follows that $b \in B$ if and only if there are $c_1,\dots,c_m \in \R$, not all zero, such that $u = \sum_{k=1}^m c_k\, w_k$ satisfies the equations on the third line of \eqref{3.4}. Since $w_k'(x_j^+) - w_k'(x_j^-) = - \delta_{jk}$, where $\delta_{jj} = 1$ and $\delta_{jk} = 0$ for $j \ne k$, this is equivalent to
\begin{equation*}
    \sum_{k=1}^m \big(w_k(x_j)\, b_j - \delta_{jk}\big)\, c_k = 0, \quad j = 1,\dots,m.
\end{equation*}
So
\begin{equation*}
    B = \bgset{b \in \R^m : \det \big(w_k(x_j)\, b_j - \delta_{jk}\big) = 0}.
\end{equation*}
This resonance set will play an important role in what follows.

First we show that the higher-order terms of $\Phi$ can be deformed away without changing the critical groups when $b \not\in B$. Let
\begin{equation*}
    \Phi_0(u) = \half \left[\int_0^1 \left((v')^2 + (w')^2\right) - \sum_{j=1}^m b_j\, w(x_j)^2\right], \quad u = v + w \in N \oplus M
\end{equation*}
be the functional associated with \eqref{3.4}.

\begin{lemma} \label{Lemma 3.1}
If \eqref{3.2}, \eqref{3.3}, and \eqref{3.5} hold, and $b \not\in B$, then zero is an isolated critical point of $\Phi$ and
\begin{equation*}
    C_q(\Phi,0) \isom C_q(\Phi_0,0) \quad \forall q.
\end{equation*}
\end{lemma}

\begin{proof}
Recall that critical groups are invariant under homotopies that preserve the isolatedness of the critical point (see Chang and Ghoussoub \cite{MR1422006} or Corvellec and Hantoute \cite{MR1926378}). Consider the homotopy
\begin{eqnarray*}
    \Phi_\tau(u) & = & (1 - \tau)\, \Phi(u) + \tau\, \Phi_0(u)\\[10pt]
    & = & \half \left[\int_0^1 (u')^2 - \sum_{j=1}^m b_j\, u(x_j)^2\right] - (1 - \tau)\, \Bigg[\int_0^1 F(x,u)\\[10pt]
    & & + \sum_{j=1}^m H_j(u(x_j))\Bigg], \quad u \in H,\, \tau \in [0,1]
\end{eqnarray*}
where $H_j(t) = \dint_0^t h_j(s)\, ds$. We will show that zero is the only critical point of $\Phi_\tau$ for all $\tau \in [0,1]$ in a sufficiently small neighborhood.

If not, there are sequences $\seq{\tau_n} \subset [0,1]$ and $\seq{u_n} \subset H \setminus \set{0}$ such that $\Phi_{\tau_n}'(u_n) = 0$ and $\rho_n := \norm{u_n} \to 0$. So, for all $y \in H$,
\begin{multline*}
    \int_0^1 u_n'\, y' - \sum_{j=1}^m b_j\, u_n(x_j)\, y(x_j) - (1 - \tau_n)\, \Bigg[\int_0^1 f(x,u_n)\, y\\[10pt]
    + \sum_{j=1}^m h_j(u_n(x_j))\, y(x_j)\Bigg] = 0.
\end{multline*}
Dividing by $\rho_n$, setting $\widetilde{u}_n := u_n/\rho_n$, and using \eqref{3.2} and \eqref{3.5} gives
\begin{equation} \label{3.6}
    \int_0^1 \widetilde{u}_n'\, y' - \sum_{j=1}^m b_j\, \widetilde{u}_n(x_j)\, y(x_j) = \o(1).
\end{equation}
Since $\seq{\widetilde{u}_n}$ is bounded in $H$, a renamed subsequence converges to some $\widetilde{u}$ weakly in $H$ and uniformly on $[0,1]$, so passing to the limit in \eqref{3.6} gives
\begin{equation*}
    \int_0^1 \widetilde{u}' y' - \sum_{j=1}^m b_j\, \widetilde{u}(x_j)\, y(x_j) = 0.
\end{equation*}
Taking $y = \widetilde{u}_n$ in \eqref{3.6}, using $\norm{\widetilde{u}_n} = 1$, and passing to the limit gives
\begin{equation*}
    \sum_{j=1}^m b_j\, \widetilde{u}(x_j)^2 = 1,
\end{equation*}
so $\widetilde{u} \ne 0$. Thus, $\widetilde{u}$ is a nontrivial solution of \eqref{3.4}, contradicting the assumption that $b \not\in B$.
\end{proof}

Next we show that the critical groups of $\Phi_0$ are the same as those of its restriction to the finite dimensional subspace $M$. Set $\Phi_b := \restr{\Phi_0}{M}$, so
\begin{equation*}
    \Phi_b(w) = \half \left[\int_0^1 (w')^2 - \sum_{j=1}^m b_j\, w(x_j)^2\right], \quad w \in M.
\end{equation*}

\begin{lemma} \label{Lemma 3.2}
We have
\begin{equation*}
    C_q(\Phi_0,0) \isom C_q(\Phi_b,0) \quad \forall q.
\end{equation*}
\end{lemma}

\begin{proof}
Taking $U = H$ in the definition \eqref{3.1} for $\Phi_0$ gives
\begin{equation*}
    C_q(\Phi_0,0) = H_q(\Phi_0^0,\Phi_0^0 \setminus \set{0}).
\end{equation*}
Consider the deformation
\begin{equation*}
    \eta(u,t) = (1 - t)\, v + w, \quad u = v + w \in N \oplus M,\, t \in [0,1].
\end{equation*}
We have
\begin{equation*}
    \Phi_0(\eta(u,t)) = \half \left[\int_0^1 \left((1 - t)^2\, (v')^2 + (w')^2\right) - \sum_{j=1}^m b_j\, w(x_j)^2\right] \le \Phi_0(u),
\end{equation*}
so $\restr{\eta}{\Phi_0^0 \times [0,1]}$ (resp. $\restr{\eta}{(\Phi_0^0 \setminus \set{0}) \times [0,1]}$) is a strong deformation retraction of $\Phi_0^0$ (resp. $\Phi_0^0 \setminus \set{0}$) onto $\Phi_0^0 \cap M = \Phi_b^0$ (resp. $(\Phi_0^0 \setminus \set{0}) \cap M = \Phi_b^0 \setminus \set{0}$). Thus,
\begin{equation*}
    C_q(\Phi_0,0) \isom H_q(\Phi_b^0,\Phi_b^0 \setminus \set{0}) = C_q(\Phi_b,0). \QED
\end{equation*}
\end{proof}

The functional $\Phi_b$ is of class $C^2$, and its Hessian at zero is given by
\begin{equation*}
    \sip{\Phi_b''(0)\, y}{z} = \int_0^1 y' z' - \sum_{j=1}^m b_j\, y(x_j)\, z(x_j), \quad y, z \in M.
\end{equation*}
So the assumption that problem \eqref{3.4} has only the trivial solution implies that zero is a nondegenerate critical point of $\Phi_b$. Let $m_0$ denote its Morse index. Since $\dim M = m$, $0 \le m_0 \le m$. With respect to the basis $\set{w_1,\dots,w_m}$ of $M$, $\Phi_b''(0)$ is represented by the $m \times m$ matrix $\big(\sip{\Phi_b''(0)\, w_j}{w_k}\big)$, which is symmetric and nonsingular, and $m_0$ is the number of negative eigenvalues of this matrix. Combining this with Lemmas \ref{Lemma 3.1} and \ref{Lemma 3.2} now gives

\begin{theorem} \label{Theorem 3.3}
If \eqref{3.2}, \eqref{3.3}, and \eqref{3.5} hold, and $b \not\in B$, then
\begin{equation*}
    C_q(\Phi,0) = \delta_{qm_0}\, \calG,
\end{equation*}
where $\calG$ is the coefficient group. In particular, $C_q(\Phi,0) = 0$ for all $q > m$.
\end{theorem}

We close this section with the observation that the critical groups of $\Phi_b$ are constant in each path-component of $\R^m \setminus B$. Indeed, if $p \in C([0,1],\R^m \setminus B)$, take any bounded neighborhood $U$ of $0$ in $M$ and
consider the homotopy
\begin{equation*}
    [0,1] \to C^1(U), \quad t \mapsto \restr{\Phi_{p(t)}}{U}.
\end{equation*}
Since zero is the only critical point of $\Phi_{p(t)}$ for all $t \in [0,1]$, it follows that $C_\ast(\Phi_{p(t)},0)$ are independent of $t$.

\section{An Application}

In this section we give an application of Theorem \ref{Theorem 3.3}.

\begin{theorem}
Assume that \eqref{2.1} -- \eqref{2.3}, \eqref{3.2}, \eqref{3.3}, and \eqref{3.5} hold, $a_j \notin \sigma_j$ for $j = 1,\dots,m+1$, and $b \not\in B$. If
\begin{equation} \label{4.1}
a_{j_0} > \lambda^{j_0}_1
\end{equation}
for some $j_0$, or
\begin{equation} \label{4.2}
    \int_0^1 (w_0')^2 \ge \sum_{j=1}^m b_j\, w_0(x_j)^2
\end{equation}
for some $w_0 \in M \setminus \set{0}$, then problem \eqref{1.1} has a nontrivial solution.
\end{theorem}

\begin{proof}
In the proof of Theorem \ref{Theorem 2.1}, the saddle point theorem actually gives a critical point $u$ with $C_k(\Phi,u) \ne 0$ where
\begin{equation*}
    k = \dim H_1 = \sum_{j \in J_1} \dim N_j^- + \dim M = \sum_{j \in J_1} d_j + m.
\end{equation*}
If \eqref{4.1} holds, then $j_0 \in J_1$ and hence $k \ge d_{j_0} + m > m$, and if \eqref{4.2} holds, then $\sip{\Phi_b''(0)\, w_0}{w_0} \ge 0$ and hence $m_0 < m \le k$. In either case, $C_k(\Phi,0) = 0$ by Theorem \ref{Theorem 3.3}, so $u \ne 0$.
\end{proof}

\begin{corollary}
Assume that \eqref{2.1} -- \eqref{2.3}, \eqref{3.2}, \eqref{3.3}, and \eqref{3.5} hold, $a_j \notin \sigma_j$ for $j = 1,\dots,m+1$, and $b \not\in B$. If
\begin{equation} \label{4.3}
b_{j_0} \le \frac{x_{j_0 + 1} - x_{j_0 - 1}}{(x_{j_0 + 1} - x_{j_0})(x_{j_0} - x_{j_0 - 1})}
\end{equation}
for some $j_0$, then problem \eqref{1.1} has a nontrivial solution.
\end{corollary}

\begin{proof}
Take $w_0$ to be the function in $M$ for which $w_0(x_j) = \delta_{jj_0}$.
\end{proof}

When the points $x_j$ are equally spaced, $\lambda^j_k = k^2\, (m + 1)^2\, \pi^2 =: \lambda_k$ and $\sigma_j = \set{\lambda_k : k = 1,2,\dots} =: \sigma$ for all $j$, and the right-hand side of \eqref{4.3} reduces to $2\, (m + 1)$, so we have

\begin{corollary}
Let $x_j = j/(m + 1),\, j = 1,\dots,m$ and assume that \eqref{2.1} -- \eqref{2.3}, \eqref{3.2}, \eqref{3.3}, and \eqref{3.5} hold, $a_j \notin \sigma$ for $j = 1,\dots,m+1$, and $b \not\in B$. If
\begin{equation*}
\max_j\, a_j > (m + 1)^2\, \pi^2,
\end{equation*}
or
\begin{equation*}
\min_j\, b_j \le 2\, (m + 1),
\end{equation*}
then problem \eqref{1.1} has a nontrivial solution.
\end{corollary}

We close with an example.

\begin{example}
Our results apply to the problem
\begin{equation*}
    \left\{\begin{aligned}
    - u'' & = \sum_{j=1}^{m+1} a_j\, \chi_j(x)\, \frac{u^3 + u^2}{u^2 + 1}, \quad x \in (0,1) \setminus \set{x_1,\dots,x_m}\\[10pt]
    u(0) & = u(1) = 0, \qquad u(x_j^+) = u(x_j^-), \quad j = 1,\dots,m\\[10pt]
    u'(x_j^+) & = u'(x_j^-) - u^3(x_j) - u^2(x_j) - b_j\, u(x_j), \quad j = 1,\dots,m.
    \end{aligned}\right.
\end{equation*}
\end{example}

{\small \def\cdprime{$''$}
}


\begin{thebibliography}{10}

\bibitem{MR2782888}
L. Bai and B. Dai.
\newblock Existence and multiplicity of solutions for an impulsive boundary
  value problem with a parameter via critical point theory.
\newblock {\em Math. Comput. Modelling}, 53(9-10):1844--1855, 2011.

\bibitem{MR1422006}
K.C. Chang and N. Ghoussoub.
\newblock The {C}onley index and the critical groups via an extension of
  {G}romoll-{M}eyer theory.
\newblock {\em Topol. Methods Nonlinear Anal.}, 7(1):77--93, 1996.

\bibitem{MR2341237}
A.A. Chikrii, I.I. Matychyn, and K.A. Chikrii.
\newblock Differential games with impulse control.
\newblock In {\em Advances in dynamic game theory}, volume~9 of {\em Ann.
  Internat. Soc. Dynam. Games}, pages 37--55. Birkh\"auser Boston, Boston, MA,
  2007.

\bibitem{MR1926378}
J.-N. Corvellec and A. Hantoute.
\newblock Homotopical stability of isolated critical points of continuous
  functionals.
\newblock {\em Set-Valued Anal.}, 10(2-3):143--164, 2002.
\newblock Calculus of variations, nonsmooth analysis and related topics.

\bibitem{MR2341245}
E. Cr{\"u}ck, M. Quincampoix, and P. Saint-Pierre.
\newblock Pursuit-evasion games with impulsive dynamics.
\newblock In {\em Advances in dynamic game theory}, volume~9 of {\em Ann.
  Internat. Soc. Dynam. Games}, pages 223--247. Birkh\"auser Boston, Boston,
  MA, 2007.

\bibitem{MR2830975}
W. Gong, Q. Zhang, and X.H. Tang.
\newblock Existence of subharmonic solutions for a class of second-order
  {$p$}-{L}aplacian systems with impulsive effects.
\newblock {\em J. Appl. Math.}, pages Art. ID 434938, 18, 2012.

\bibitem{MR2781070}
Z. Han and S. Wang.
\newblock Mixed two-point boundary-value problems for impulsive differential
  equations.
\newblock {\em Electron. J. Differential Equations}, pages No. 35, 14, 2011.

\bibitem{MR2241134}
X. Lin and D. Jiang.
\newblock Multiple positive solutions of {D}irichlet boundary value problems
  for second order impulsive differential equations.
\newblock {\em J. Math. Anal. Appl.}, 321(2):501--514, 2006.

\bibitem{MR1487265}
X. Liu and D. Guo.
\newblock Periodic boundary value problems for a class of second-order
  impulsive integro-differential equations in {B}anach spaces.
\newblock {\em J. Math. Anal. Appl.}, 216(1):284--302, 1997.

\bibitem{MR2474254}
J.J. Nieto and D. O'Regan.
\newblock Variational approach to impulsive differential equations.
\newblock {\em Nonlinear Anal. Real World Appl.}, 10(2):680--690, 2009.

\bibitem{MR1756756}
L. Stone, B. Shulgin, and Z. Agur.
\newblock Theoretical examination of the pulse vaccination policy in the {SIR}
  epidemic model.
\newblock {\em Math. Comput. Modelling}, 31(4-5):207--215, 2000.

\bibitem{MR2465922}
Y. Tian and W. Ge.
\newblock Applications of variational methods to boundary-value problem for
  impulsive differential equations.
\newblock {\em Proc. Edinb. Math. Soc. (2)}, 51(2):509--527, 2008.

\bibitem{MR2570525}
H. Zhang and Z. Li.
\newblock Variational approach to impulsive differential equations with
  periodic boundary conditions.
\newblock {\em Nonlinear Anal. Real World Appl.}, 11(1):67--78, 2010.

\bibitem{MR1962293}
X. Zhang, Z. Shuai, and K. Wang.
\newblock Optimal impulsive harvesting policy for single population.
\newblock {\em Nonlinear Anal. Real World Appl.}, 4(4):639--651, 2003.

\bibitem{MR2570535}
Z. Zhang and R. Yuan.
\newblock An application of variational methods to {D}irichlet boundary value
  problem with impulses.
\newblock {\em Nonlinear Anal. Real World Appl.}, 11(1):155--162, 2010.

\bibitem{MR2532812}
J. Zhou and Y. Li.
\newblock Existence and multiplicity of solutions for some {D}irichlet problems
  with impulsive effects.
\newblock {\em Nonlinear Anal.}, 71(7-8):2856--2865, 2009.

\end{thebibliography}
\end{document}